\documentclass[11pt,leqno]{amsart}%
\usepackage{amsmath}
\usepackage{amsfonts}
\usepackage{amssymb}
\usepackage{graphicx}
\usepackage[margin=3.30cm]{geometry}
\providecommand{\U}[1]{\protect\rule{.1in}{.1in}}
\newtheorem{theorem}{Theorem}[section]
\newtheorem*{acknowledgement*}{Acknowledgement}

\newtheorem{example}[theorem]{Example}

\newtheorem{lemma}[theorem]{Lemma}

\newtheorem{proposition}[theorem]{Proposition}
\newtheorem{remark}[theorem]{Remark}

\newtheorem{thm}{Theorem}

\begin{document}
\title[Translating solitons of the Mean Curvature Flow]{Rigidity results and topology at infinity of translating solitons of the Mean Curvature Flow}


\date{\today}

\author {Debora Impera}
\address{Universit\`a degli Studi di Milano Bicocca\\
Dipartimento di Matematica e Applicazioni\\
via Cozzi 55\\
I-20125 Milano, ITALY}
\email{debora.impera@gmail.com}
\author {Michele Rimoldi}
\address{Universit\'e Paris 13\\
Sorbonne Paris Cit\'e\\
LAGA\\
CNRS (UMR 7539)\\
99, Avenue Jean-Baptiste Cl\'ement\\
F-93430 Villetaneuse, FRANCE}
\email{michele.rimoldi@gmail.com}

\begin{abstract}
In this paper we obtain rigidity results and obstructions on the topology at infinity of translating solitons of the mean curvature flow in the Euclidean space. Our approach relies on the theory of $f$--minimal hypersurfaces.
\end{abstract}

\maketitle

\tableofcontents

\section{Introduction}\label{intro}
By a translating soliton of the mean curvature flow (\textsl{translator} for short) we mean a connected isometrically immersed complete hypersurface $x:\Sigma^m\to\mathbb{R}^{m+1}$ whose mean curvature vector field $\mathbf{H}$ satisfies the equation
\begin{equation}\label{TranslEq}
\mathbf{H}=v^{\bot}
\end{equation}
for some fixed unit length vector $v\in\mathbb{R}^{m+1}$, where $(\cdot)^{\bot}$ denotes the projection on the normal bundle of $\Sigma$. Note that here we are using the convention
\[
 \mathbf{H}=H\nu=\mathrm{tr}(\mathbf{A}),
\]
where $\nu$ is the unit normal vector of $\Sigma$, $H$ is the mean curvature function, and the second fundamental form of the immersion is defined as the generalized Hessian
\[
 \mathbf{A}=Ddx.
\]
When possible, we will choose the unit normal $\nu$ to be inward pointing, and we will agree to say that $\Sigma$ is mean convex if it holds that $H\geq0$.

Solutions of \eqref{TranslEq} correspond to translating solutions $\{\Sigma_t=\Sigma+tv\}_{t\in\mathbb{R}}$ of the mean curvature flow, and play a key role in the study of slowly forming singularities; see e.g. \cite{HS1}, \cite{HS2}. Without loss of generality we can assume that the velocity vector $v$ is given by $v=e_{m+1}$, where $\{e_1,\ldots,e_{m+1}\}$ is the standard orthonormal basis of $\mathbb{R}^{m+1}$.
\medskip
As a first basic example, note that if a translator has zero mean curvature, by \eqref{TranslEq} we have that $v$ must be tangential to the translator. Consequently, translators which are also minimal hypersurfaces split as the product $\tilde{\Sigma}\times L$, where $L$ is a line parallel to $v$ and $\tilde{\Sigma}$ is a minimal hypersurface in $L^{\bot}$.

Here are some other typical examples of translators. 

\begin{example}(Grim reaper and grim reaper cylinder)
 \rm{In the case of $\gamma:I\to\mathbb{R}^2$, smooth curve parametrized by its arc--length, \eqref{TranslEq} becomes
 \[
  \ddot{\gamma}=\left\langle v, R\dot{\gamma}\right\rangle R\dot{\gamma},
 \]
 where $R:\mathbb{R}^2\to\mathbb{R}^2$ is the counterclockwise rotation of $\frac{\pi}{2}$. By integration of this ODE, one can prove that the only possible translating curve with non--identically zero curvature is given by the graph of $y=-\log(\cos(x))$, where $x\in(-\frac{\pi}{2}, \frac{\pi}{2})$. This curve is called the \textsl{grim reaper}, \cite{Grayson}.

Taking the orthogonal product of a grim reaper with $\mathbb{R}^{m-1}$, one easily obtains another (higher dimensional) example of a translator in $\mathbb{R}^{m+1}$, the \textsl{grim reaper cylinder}. It can be easily proved that these examples are mean convex. Indeed, they have only one strictly positive principal curvature.}

\end{example}

\begin{example}\label{BowlTC}(Bowl soliton and the translating catenoid)
\rm{ In \cite{AW} it is proved that there exist a unique (up to rigid motion) solution of \eqref{TranslEq} which is rotationally symmetric and strictly convex. When $m=1$ this is the grim reaper. When $m>1$, the solution (which is actually an entire graph growing quadratically at infinity) roughly looks like a paraboloid, and is usually called the \textsl{bowl soliton}.

In \cite{XJWang}, X.--J. Wang proved that in dimension $m=2$ any entire convex translator must be rotationally symmetric, and hence the bowl soliton. However, in \cite{XJWang} it is also shown that for $m\geq 3$ there exist entire strictly convex translators that are \textsl{not rotationally symmetric}. See also the very recent \cite{Has} by R. Haslhofer, where the uniqueness of the bowl soliton is proved in arbitrary dimension under somewhat different assumptions.

Without assuming convexity, it was proved by J. Clutterbuck, O. Schn\"urer and F. Schulze, \cite{CSS}, that rotationally symmetric translators coincide (up to rigid motion) either with the bowl soliton or with the \textsl{translating catenoid}, a complete non--convex translator made up of the union of two graphical ``winglike'' solutions in the complement of a ball that are asymptotic at infinity to a bowl soliton.}
\end{example}

\begin{example}
\rm{By the desingularization technique, X.--H. Nguyen, \cite{Ng1, Ng2}, exhibited new examples of translators. In particular, note that in these papers $2$--dimensional translators with infinite genus are constructed.}
\end{example}

As we shall see in a moment, translators of the mean curvature flow turn out to belong to the more general class of $f$--minimal hypersurfaces. This class of hypersurfaces has been extensively studied in recent years due to the fact that, besides minimal hypersurfaces, also self--shrinkers of the mean curvature flow pertain to it; see for instance \cite{ChMeZh1}, \cite{Esp}, \cite{Fan}, \cite{Ho}, \cite{IR}, \cite{Liu}. On the other hand, the general theory developed for $f$--minimal hypersurfaces has not been exploited yet for the study of translators. The aim of this paper is to highlight how the realm of weighted manifolds and $f$--minimal hypersurfaces can naturally give strong enough characterization and topological results for translators. 

\medskip

Recall that a weighted manifold is a triple $M_f^{m+1}=\left(M^{m+1}, \left\langle\,,\,\right\rangle, e^{-f}d\mathrm{vol}_{M}\right)$ , where $\left(M^{m+1}, \left\langle\,,\,\right\rangle\right)$ is a complete $(m+1)$-dimensional Riemannian manifold, $f\in C^{\infty}(M)$, and $d\mathrm{vol}_{M}$ denotes the canonical Riemannian volume measure on $M$. Following Gromov, \cite{Gro}, if we consider an isometrically immersed orientable hypersurface $\Sigma^m$  in the weighted manifold $M_f$, the weighted mean curvature vector field is defined as
\[
 \mathbf{H}_f=\mathbf{H}+(\overline{\nabla} f)^\bot.
\]
Here we have denoted by $\overline{\nabla}$ the Levi--Civita connection on $M$. From variatonal formulae, \cite{Bay}, one can see that $\Sigma$ is $f$--minimal, namely a critical point of the weighted area functional
\[
 \mathrm{vol}_f(\Sigma)=\int_{\Sigma}e^{-f}d\mathrm{vol}_{\Sigma},
\]
if and only if $\mathbf{H}_{f}$ vanishes identically.

Let $x:\Sigma^m\to\mathbb{R}^{m+1}$ be  a translator satisfying \eqref{TranslEq} for some $v\in\mathbb{R}^{m+1}$. Then, letting $f=-\left\langle v, x\right\rangle$, it is easy to see that $\Sigma$ is a $f$--minimal hypersurfaces in $\mathbb{R}^{m+1}$. Moreover, looking at the Bakry--\'Emery Ricci tensor of $\Sigma$, we have that
\begin{equation}\label{CurvBnd}
 \mathrm{Ric}_f(X,X)=\mathrm{Ric}(X, X)+\mathrm{Hess}(f)(X, X)=-\left\langle A^2 X, X\right\rangle,
\end{equation}
for any $X\in T\Sigma$; see e.g. Section 7 in \cite{IR}.

Computing the second variation formula for the weighted area functional, it can be seen that stability properties of $f$--minimal hypersurfaces are related to spectral properties of the weighted Jacobi operator
\[
 L_f=-\Delta_f-\left(|\mathbf{A}|^2+\overline{\mathrm{Ric}}_f(\nu, \nu)\right),
\]
where $\overline{\mathrm{Ric}}_f$ denotes the Bakry--\'Emery Ricci tensor of the ambient space and $\Delta_f=\Delta-\left\langle\nabla f, \,\nabla\cdot\right\rangle$ is the $f$--Laplacian operator on $\Sigma_f$. In particular, we say that a $f$--minimal hypersurface $\Sigma$ is $f$--stable if for any compactly supported smooth function $u\in C_{c}^{\infty}(\Sigma)$, it holds that
\[
 \int_{\Sigma}uL_{f}u\,e^{-f}d\mathrm{vol}_{\Sigma}\geq 0.
\]
The $f$--index of $\Sigma$ is the generalized Morse index of $L_f$ on $\Sigma$. Namely, given a bounded domain $\Omega\subset\Sigma$, we define
\[
 Ind^{L_f}(\Omega)=\#\{\mathrm{negative\,\,eigenvalues\,\,of}\,\, L_f\,\mathrm{on}\,C_{0}^{\infty}\}.
\]
The $f$--index of $\Sigma$ is then defined as
\[
 Ind_{f}(\Sigma):=Ind^{L_f}(\Sigma)=\sup_{\Omega\subset\subset\Sigma}Ind^{L_f}(\Omega).
\]

In the  special case of $x:\Sigma^m\to\mathbb{R}^{m+1}$ translator of the mean curvature flow, the $f$--stability operator takes the form
\[
 L_f=-\Delta_f-|\mathbf{A}|^2.
\]
It is not difficult to prove that translator hyperplanes (i.e. hyperplanes containing the direction $v$) are $f$--stable.
Moreover, it has been proved by C. Arezzo and J. Sun, \cite{AS}, but also by L. Ma, using different techniques, \cite{LMa}, that the grim reaper is a $f$--stable translator in $\mathbb{R}^2$. Similar computations show that also the grim reaper cylinders are $f$--stable. More generally, note that L. Shahriyari has proved that translating graphs are $f$--stable, \cite{Sh}.

The first main result of the paper is the following rigidity theorem for $f$--stable translators under a weighted $L^2$--condition on the norm of the second fundamental form.

\begin{thm}\label{MainA}
Let $x:\Sigma^{m\geq 2}\to\mathbb{R}^{m+1}$ be a $f$--stable translator of the mean curvature flow. Assume that $|\mathbf{A}|\in L^2(\Sigma_f)$. Then $\Sigma$ is a translator hyperplane.
\end{thm}

An application of the maximum principle and the weighted version of a result in \cite{FCS} give that translators with mean curvature that does not change sign are either $f$--stable (generalizing in particular Theorem 1.2.5 in \cite{Sh}) or they split as the product of a line parallel to the translating direction and a minimal hypersurface in the orthogonal complement of the line. Note that, in this latter case, by Fubini's theorem, the condition $|\mathbf{A}|\in L^{2}(\Sigma_f)$ is met if and only if $|\mathbf{A}|\equiv 0$ (i.e. $\Sigma$ is a translator hyperplane). Hence the conclusion of Theorem \ref{MainA} still holds in this situation. On the other hand, under this assumption we are able to strengthen the result as follows, improving Theorem 6 in \cite{MaMi}. 

\begin{thm}\label{MainA1}
Let $x:\Sigma^{m\geq 2}\to\mathbb{R}^{m+1}$ be a translator with mean curvature which does not change sign. Suppose that the traceless second fundamental form of the immersion $\mathbf{\Phi}=\mathbf{A}-\frac{H}{m}\mathbf{Id}$ satisfies  $|\mathbf{\Phi}|\in L^2(\Sigma_f)$. Then $\Sigma$ is a translator hyperplane.
\end{thm}

In the second part of the paper we focus our attention on topological properties of translators.

Exploiting in an essential way the correspondence, discovered by K. Smoczyk, \cite{Smoc}, between translators of the mean curvature flow in $\mathbb{R}^{m+1}$ and minimal hypersurfaces in the manifold $\mathbb{R}^{m+1}\times\mathbb{R}$, endowed with a suitable warped product metric, we obtain the validity  of a weighted $(m+1)$--dimensional $L^1$ Sobolev inequality on translators. The validity of this inequality permits to obtain very neat results on the topology at infinity of $f$--stable translators and translator with finite $f$--index. The main tools here are weighted versions of the Li--Tam theory and of an abstract finiteness result. For more details see \cite{IR}.

\begin{thm}\label{MainD}
Let $x:\Sigma^{m\geq 2}\to\mathbb{R}^{m+1}$ be a $f$--stable translator. Then $\Sigma$ has at most one end.
\end{thm}

\begin{remark}
\rm{
As observed above, Theorem \ref{MainD} in particular applies to translators with $|H|>0$. Moreover, in case $H\equiv 0$, i.e. $\Sigma=\tilde{\Sigma}\times\mathbb{R}$ with $\tilde{\Sigma}$ minimal, $\Sigma$ has only one end (see e.g. Lemma 9.5 in \cite{PRS-Book}). This permits to deduce that, for $m\geq 2$, if $H$ does not change sign, then the translator has at most one end.}
\end{remark}
\begin{remark}
\rm{
As a straightforward consequence of Theorem \ref{MainD} we get that the translating catenoid defined in Example \ref{BowlTC} is an example of $f$-unstable translator. 
}
\end{remark}

Concerning the topology at infinity of translators with finite $f$--index we obtain the following

\begin{thm}\label{MainB}
Let $x:\Sigma^{m\geq 2}\to\mathbb{R}^{m+1}$ be a translator with finite $f$--index. Then $\Sigma$ has finitely many ends.
\end{thm}

Under the geometric assumption that the translator is contained in a halfspace determined by an hyperplane orthogonal to the translating direction $v$, we are able to guarantee that the standard ($m$--dimensional) $L^1$ Sobolev inequality holds. Hence, as a consequence of Proposition 8 in \cite{IR},  the same conclusion of Theorem \ref{MainB} holds replacing the finiteness of the $f$--index by the integrability condition $|\mathbf{A}|\in L^{m}(\Sigma_f)$. Actually, in this situation we are able to prove the following stronger result. 

\begin{thm}\label{MainC}
Let $x:\Sigma^m\to\mathbb{R}^{m+1}$, $m\geq3$, be a translator contained in the halfspace $\Pi_{v,a}:=\{p\in \mathbb{R}^{m+1}:\langle p,v\rangle \geq a\}$, for some $a\in\mathbb{R}$. Assume that $|\mathbf{A}|\in L^m(\Sigma_f)$. Then $\Sigma$ is properly immersed and has finite topological type, i.e., there exists a smooth compact subset $\Omega\subset\subset\Sigma$ such that $\Sigma\setminus\Omega$ is diffeomorphic to the half--cylinder $\partial\Omega\times[0,+\infty)$.
\end{thm}

The paper is organized as follows. In Section \ref{SectBasicEq} we collect basic equations that we shall use in the proofs of our results. Section \ref{SectRigRes} is devoted to the proof of Theorem  \ref{MainA} and Theorem \ref{MainA1}. In Section \ref{SectSobIneq} we derive the weighted $L^1$ Sobolev inequalities for translators that will be employed in Section \ref{SectTopCons} to prove Theorems \ref{MainD}, \ref{MainB}, and \ref{MainC}. A final section is devoted to some remarks on boundedness properties of translators.

\section{Basic equations}\label{SectBasicEq}
In the following lemma we collect some useful identities for basic geometric quantities on translators, naturally involving the weighted Laplacian $\Delta_{f}$.

\begin{lemma}
Let $x:\Sigma^m\to\mathbb{R}^{m+1}$ be a translator of the mean curvature flow, set $f=-\left\langle x, v\right\rangle$, and let $\mathbf{\Phi}=\mathbf{A}-\frac{H}{m}\mathbf{Id}$ be the traceless second fundamental form. Then
\begin{align}
\Delta_{f}f=&-1,\label{Eqf}\\
\Delta_{f}|x|^2=&2\left(m+\left\langle x, v\right\rangle\right),\label{Eqx2}\\
\Delta_{f}H=&-|\mathbf{A}|^2H,\label{EqMeanCurv}\\
\Delta_{f}|\mathbf{A}|^2=&2|\nabla \mathbf{A}|^2-2|\mathbf{A}|^4,\label{EqSimon}\\
\Delta_{f}|\mathbf{\Phi}|^2=&2|\nabla\mathbf{\Phi}|^2-2|\mathbf{A}|^2|\mathbf{\Phi}|^2.\label{EqSimonTraceless}
\end{align}
\end{lemma}

\begin{proof}
To prove \eqref{Eqf}, letting $Y\in T\Sigma$, we compute
\begin{align*}
\left\langle \nabla f, Y\right\rangle=&-Y\left\langle x,v\right\rangle=-\left\langle \overline{\nabla}_{Y}x, v\right\rangle\\
=&-\left\langle Y,v\right\rangle.
\end{align*}
Hence, by \eqref{TranslEq}, 
\[
\ \nabla f=-v^{T}=-v+\left\langle v, \nu\right\rangle\nu=-v+\mathbf{H}.
\]
Letting now $Y, Z\in T\Sigma$, we have that
\begin{align*}
\mathrm{Hess}\,f(Y,Z)=&\left\langle \nabla_{Y}\nabla f, Z\right\rangle=\left\langle \overline{\nabla}_{Y}\nabla f, Z\right\rangle\\
=&-\left\langle \overline{\nabla}_{Y}v, Z\right\rangle+\left\langle v, \nu\right\rangle\left\langle \overline{\nabla}_{Y}\nu, Z\right\rangle\\
=&-\left\langle v, \nu\right\rangle\left\langle A Y, Z\right\rangle\\
=&-H\left\langle AY, Z\right\rangle.
\end{align*}
Taking the trace of this latter and using \eqref{TranslEq}, we obtain that
\begin{equation*}
\Delta_{f}f=\Delta f-|\nabla f|^2=-H^2-|v^{T}|^2=-1,
\end{equation*}
that is, \eqref{Eqf}.

As for \eqref{Eqx2}, with a similar computation, we get that for $Y\in T\Sigma$,
\begin{equation*}
\left\langle \nabla|x|^2, Y\right\rangle=Y|x|^2=2\left\langle \overline{\nabla}_{Y}x,x\right\rangle=2\left\langle Y,x\right\rangle,
\end{equation*}
i.e.,
\[
\ \nabla|x|^2=2x^{T}=2x-2\left\langle x, \nu\right\rangle\nu.
\]
Thus, we obtain that for $Y, Z\in T\Sigma$,
\begin{align*}
\mathrm{Hess}\,|x|^2(Y,Z)=&\left\langle \nabla_{Y}\nabla|x|^2,Z\right\rangle=2\left\langle \overline{\nabla}_{Y}\left(x-\left\langle x, \nu\right\rangle\nu\right), Z\right\rangle\\
=&2\left\langle Y, Z\right\rangle-2\left\langle x, \nu\right\rangle\left\langle \overline{\nabla}_{Y}\nu, Z\right\rangle\\
=&2\left(\left\langle Y, Z\right\rangle+\left\langle x, \nu\right\rangle\left\langle AY, Z\right\rangle\right).
\end{align*}
Taking the trace in the previous equation we get
\[
\ \Delta|x|^2=2\left(m+H\left\langle x, \nu\right\rangle\right),
\]
and hence,
\begin{align*}
\Delta_{f}|x|^2=&\Delta|x|^2-\left\langle \nabla f, \nabla |x|^2\right\rangle=2\left(m+H\left\langle x, \nu\right\rangle+\left\langle v^T, x^T\right\rangle\right)\\
=&2\left(m+H\left\langle x, \nu\right\rangle+\left\langle v, x\right\rangle-\left\langle v, \nu\right\rangle\left\langle x, \nu\right\rangle\right)\\
=&2\left(m+\left\langle v, x\right\rangle\right),
\end{align*}
that is, \eqref{Eqx2}.

Analogously, 
\begin{equation*}
Y(H)=\left\langle \overline{\nabla}_{Y} \nu,v\right\rangle=-\left\langle AY, v\right\rangle=-\left\langle Av^T, Y\right\rangle,
\end{equation*}
for any $Y\in T\Sigma$, i.e., $\nabla H=-Av^T$. Hence, for $Y, Z\in T\Sigma$,
\begin{align*}
\mathrm{Hess}\,H(Y, Z)=&\left\langle \nabla_{Y}\nabla H, Z\right\rangle=-\left\langle \nabla_{Y}Av^T, Z\right\rangle\\
=&-\left\langle \left(\nabla_{Y}A\right)v^T, Z\right\rangle- \left\langle A\nabla_{Y}v^T, Z\right\rangle.
\end{align*}
Using Codazzi's equation, we get
\begin{align*}
\mathrm{Hess}\,H(Y, Z)=&-\left\langle \left(\nabla_{v^T}A\right)Y, Z\right\rangle-\left\langle \overline\nabla_{Y}v^T, AZ\right\rangle\\
=&-\left\langle \left(\nabla_{v^T}A\right)Y, Z\right\rangle-\left\langle \overline{\nabla}_{Y}\left(v-H\nu\right), AZ\right\rangle\\
=&-\left\langle \left(\nabla_{v^T}A\right)Y, Z\right\rangle+H\left\langle \overline{\nabla}_{Y}\nu,AZ\right\rangle\\
=&-\left\langle \left(\nabla_{v^T}A\right)Y, Z\right\rangle-H\left\langle A^2Y,Z\right\rangle.
\end{align*}
Taking the trace and using the relation
\[
\ \mathrm{tr}\left(\nabla_{v^T}A\right)=\left\langle v^T, \nabla\left(\mathrm{tr}A\right)\right\rangle,
\]
we obtain \eqref{EqMeanCurv}.

As for the proof of \eqref{EqSimon}, we refer the reader to \cite{CMZ-Simon}, taking into account that the ambient space is $\mathbb{R}^{m+1}$ and $f=-\left\langle v, x\right\rangle$.

Finally, in order to obtain \eqref{EqSimonTraceless}, note that 
\begin{equation}\label{NormPhi2}
|\mathbf{\Phi}|^2=|\mathbf{A}|^2
-\frac{H^2}{m}.
\end{equation}
Hence, by \eqref{EqSimon} and \eqref{EqMeanCurv},
\begin{align*}
\frac{1}{2}\Delta_{f}|\mathbf{\Phi}|^2=&\frac{1}{2}\Delta_{f}|\mathbf{A}|^2-\frac{1}{2m}\Delta_{f}H^2\\
=&|\nabla\mathbf{A}|^2-|\mathbf{A}|^4-\frac{1}{m}\left(H\Delta_{f}H+|\nabla H|^2\right)\\
=&|\nabla\mathbf{A}|^2-\frac{1}{m}|\nabla H|^2-|\mathbf{A}|^2\left(|\mathbf{A}|^2-\frac{H^2}{m}\right)\\
=&|\nabla\mathbf{\Phi}|^2-|\mathbf{A}|^2|\mathbf{\Phi}|^2.
\end{align*}

\end{proof}

\section{Rigidity results}\label{SectRigRes}
The proof of Theorem \ref{MainA} relies on the following two lemmas. The first is a vanishing result proved in \cite{R1}, which adapts to the weighted setting a result originally obtained in \cite{PRS-JFA05}, \cite{PRS-Book}; see also \cite{PV1}.

\begin{lemma}[Theorem 8 in \cite{R1}]\label{Liouville}
Assume that on a weighted manifold $M_f$ the locally Lipschitz functions $u\geq 0$, $v>0$ satisfy
\begin{equation}\label{ineq_u}
\Delta_fu+a(x)u\geq0
\end{equation}
and
\begin{equation}\label{ineq_deltav}
\Delta_fv+\delta a(x)v\leq0,
\end{equation}
for some constant $\delta\geq1$ and $a(x)\in C^0(M)$. If $u\in L^{2\beta}\left(M_f\right)$, $1\leq\beta\leq \delta$, then there exists a constant $C\geq0$ such that 
\[
\ u^{\delta}=Cv.
\]
Furthermore,
\begin{enumerate}
	\item [(i)]If $\delta>1$ then $u$ is constant on $M$ and either $a\equiv 0$ or $u\equiv 0$;
	\item [(ii)]If $\delta=1$ and $u\not\equiv0$, $v$ and therefore $u^\delta$ satisfy \eqref{ineq_deltav} with equality sign.
\end{enumerate}
\end{lemma}

The second preliminary result we need in the proof of Theorem \ref{MainA} is the following

\begin{lemma}\label{Lemma2}
Let $x:\Sigma^m\to\mathbb{R}^{m+1}$ be a $f$--stable translator with $H(p)\neq 0$ for some $p\in\Sigma$, and let $\omega\in C^{2}(\Sigma)$ be a positive solution of the stability equation
\[
\ \Delta_{f}\omega+|\mathbf{A}|^2\omega=0.
\]
If $|\mathbf{A}|\in L^2(\Sigma_{f})$, then there exists a constant $C\in\mathbb{R}\setminus\left\{0\right\}$ such that $H=C\omega$. In particular $|H|>0$.
\end{lemma}

\begin{proof}
We follow the argument used in the proof of Lemma 9.25 in \cite{CoMi}. \\
First of all, note that the existence of $\omega$ is guaranteed by the $f$--stability of $\Sigma$, using the weighted version of a result by D. Fischer--Colbrie and R. Schoen, \cite{FCS}, (see e.g. Proposition 3 in \cite{IR}). Moreover, recall that, if $u, v\in C^{2}(\Sigma)$ satisfy
\begin{equation}\label{IntConduv}
\int_{\Sigma}\left(|u\nabla v|+|\nabla u||\nabla v|+|u\Delta_{f}v|\right)e^{-f}d\mathrm{vol}_{\Sigma}<+\infty,
\end{equation} 
then
\[
\ \int_{\Sigma}u\Delta_{f}v\,e^{-f}d\mathrm{vol}_{\Sigma}=-\int_{\Sigma}\left\langle\nabla u, \nabla v\right\rangle e^{-f}d\mathrm{vol}_{\Sigma}.
\]
Since $|\mathbf{A}|\in L^2(\Sigma_{f})$, we have that $H\in W^{1,2}(\Sigma_{f})$. Furthermore, it follows by \eqref{TranslEq}, that $|H|\leq 1$, and thus
\[
\ |H\Delta_{f}H|=|\mathbf{A}|^2H^2\leq |\mathbf{A}|^2.
\]
In particular, Equation \eqref{IntConduv} is satisfied for $u=v=H$, and we get that
\begin{equation}\label{IntCondHA}
\int_{\Sigma}H^2|\mathbf{A}|^2e^{-f}d\mathrm{vol}_{\Sigma}=-\int_{\Sigma}H\Delta_{f}He^{-f}d\mathrm{vol}_{\Sigma}=\int_{\Sigma}|\nabla H|^2
e^{-f}d\mathrm{vol}{\Sigma}.
\end{equation}
Let us now prove that $H|\nabla \log \omega|\in L^{2}(\Sigma_{f})$. Given a  compactly supported function $\eta$ on $\Sigma$ it holds that
\begin{align*}
\ \int_{\Sigma}\left\langle\nabla \eta^2, \nabla \log\omega \right\rangle e^{-f}d\mathrm{vol}_{\Sigma}&=-\int_{\Sigma}\eta^{2}\Delta_{f}\log\omega \,e^{-f}d\mathrm{vol}_{\Sigma}\\
&=\int_{\Sigma}\eta^2\left(|\mathbf{A}|^2+|\nabla\log\omega|^2\right)e^{-f}d\mathrm{vol}_{\Sigma}.
\end{align*}
On the other hand 
\[
\ \int_{\Sigma}\left\langle\nabla \eta^2, \nabla\log \omega\right\rangle e^{-f}d\mathrm{vol}_{\Sigma}\leq\int_{\Sigma}\left[\frac{\eta^2}{2}|\nabla \log\omega|^2+2|\nabla \eta|^2\right]e^{-f}d\mathrm{vol}_{\Sigma}.
\]
Inserting this in the above identity yields
\[
\ \int_{\Sigma}\eta^2\left(|\mathbf{A}|^2+|\nabla\log\omega|^2\right)e^{-f}d\mathrm{vol}_{\Sigma}\leq 4\int_{\Sigma}|\nabla\eta|^2e^{-f}
d\mathrm{vol}_{\Sigma}.
\]
Now let
\[
\ \eta_{j}=
\begin{cases}
1&\mathrm{on}\,\, B_{j}(o)\\
j+1-r(x)&\mathrm{on}\,\,B_{j+1}(o)\setminus B_{j}(o)\\
0&\mathrm{on}\,\,\Sigma\setminus B_{j+1}(o),
\end{cases}
\]
where $o\in\Sigma$ is a fixed reference point and $r$ is the distance function from $o$. Setting $\eta=\eta_{j}\,H$, we have that
\begin{align*}
\int_{\Sigma}\eta^2\left(|\mathbf{A}|^2+|\nabla\log\omega|^2\right)e^{-f}d\mathrm{vol}_{\Sigma}\leq&4\int_{\Sigma}|\nabla(\eta_{j}H)|^2e^{-f}d\mathrm{vol}_{\Sigma}\\
\leq&8\int_{\Sigma}H^2|\nabla\eta_{j}|^2e^{-f}d\mathrm{vol}_{\Sigma}+8\int_{\Sigma}\eta_{j}^2|\nabla H|^2 e^{-f}d\mathrm{vol}_{\Sigma}\\
\leq&8\int_{\Sigma}\left(H^2+|\nabla H|^2\right)e^{-f}d\mathrm{vol}_{\Sigma}<+\infty.
\end{align*}
Taking the limit as $j\to+\infty$ and using the dominated convergence theorem, we obtain
\[
\ \int_{\Sigma}H^2|\nabla\log\omega|^2e^{-f}d\mathrm{vol}_{\Sigma}\leq\int_{\Sigma}H^2\left(|\mathbf{A}|^2+|\nabla \log \omega|^2\right)e^{-f}d\mathrm{vol}_{\Sigma}<+\infty.
\]
Hence, using H\"older's inequality, we get
\begin{align*} \int_{\Sigma}H^2|\nabla\log\omega|e^{-f}d\mathrm{vol}_{\Sigma}\leq&\left(\int_{\Sigma}H^2e^{-f}d\mathrm{vol}_{\Sigma}\right)^{\frac{1}{2}}\left(\int_{\Sigma}H^2|\nabla\log\omega|^2e^{-f}d\mathrm{vol}_{\Sigma}\right)^{\frac{1}{2}}<+\infty,\\
\int_{\Sigma}|\nabla H^2||\nabla\log\omega|e^{-f}d\mathrm{vol}_{\Sigma}=&2\int_{\Sigma}|H||\nabla H||\nabla \log \omega|e^{-f}d\mathrm{vol}_{\Sigma}\\
\leq&\left(\int_{\Sigma}|\nabla H|^2e^{-f}d\mathrm{vol}_{\Sigma}\right)^{\frac{1}{2}}\left(\int_{\Sigma}H^2|\nabla\log\omega|^2e^{-f}d\mathrm{vol}_{\Sigma}\right)^{\frac{1}{2}}\\
<&+\infty.
\end{align*}
Finally,
\begin{equation*}
\int_{\Sigma}|H^2\Delta_{f}\log\omega|e^{-f}d\mathrm{vol}_{\Sigma}\leq\int_{\Sigma}H^2\left(|\mathbf{A}|^2+|\nabla\log\omega|^2|\right)e^{-f}d\mathrm{vol}_{\Sigma}<+\infty.
\end{equation*}
Hence \eqref{IntConduv} is satisfied with the choices $u=H^2$, $v=\log \omega$ and we deduce that
\begin{align*}
\int_{\Sigma}\left\langle\nabla H^2, \nabla\log\omega\right\rangle\,e^{-f}d\mathrm{vol}_{\Sigma}=&-\int_{\Sigma}H^2\Delta_{f}\log\omega\,e^{-f}d\mathrm{vol}_{\Sigma}\\
=&\int_{\Sigma}\left(H^2|\mathbf{A}|^2+H^2|\nabla\log\omega|^2\right)\,e^{-f}d\mathrm{vol}_{\Sigma}.
\end{align*}
Substituting \eqref{IntCondHA} in the previous equation we get
\begin{align*}
0=&\int_{\Sigma}\left(|\nabla H|^2+H^2|\nabla\log\omega|^2-\left\langle\nabla H^2, \nabla \log\omega\right\rangle\right)e^{-f}d\mathrm{vol}_{\Sigma}\\
=&\int_{\Sigma}|\nabla H-H\nabla\log\omega|^2e^{-f}d\mathrm{vol}_{\Sigma}.
\end{align*}
Hence $\nabla H= H\,\nabla\log\omega$, implying that $\frac{H}{\omega}$ has to be constant.
\end{proof}

\begin{proof}(of Theorem \ref{MainA})
By \eqref{EqSimon} we have that
\begin{equation}\label{EqVanu}
|\mathbf{A}|(\Delta_f+|\mathbf{A}|^2)|\mathbf{A}|=|\nabla\mathbf{A}|^2-|\nabla|\mathbf{A}||^2\geq0.
\end{equation}
Moreover, since $\Sigma$ is $f$--stable, there exists a positive $C^2$ solution $\omega$ of
\begin{equation}\label{EqVanv}
\Delta_f\omega+|\mathbf{A}|^2\omega=0.
\end{equation}
We can then apply Lemma \ref{Liouville} with the choices $u=|\mathbf{A}|$, $v=\omega$, $a(x)=|\mathbf{A}|^2$ and $\delta=\beta=1$, to deduce that
\begin{equation}\label{Rel1}
|\mathbf{A}|=C_1\omega,
\end{equation}
for some constant $C_1\geq 0$, and that either $|A|\equiv0$ and $\Sigma$ is necessarily a translator hyperplane, or
\begin{equation}\label{Rel2}
|\nabla \mathbf{A}|^2=|\nabla|\mathbf{A}||^2.
\end{equation}
In case $H\equiv 0$, combining \eqref{Rel2} with Lemma 10.2 in \cite{CoMi}, we get the constancy of $|\mathbf{A}|$. Hence, by \eqref{EqSimon}, $|\mathbf{A}|\equiv 0$, but we have already analyzed this situation. Therefore, by Lemma \ref{Lemma2}, we know that $|H|>0$.

Moreover, combining Lemma \ref{Lemma2} with \eqref{Rel1}, we have that
\begin{equation}\label{Rel3}
|\mathbf{A}|=C_2H,
\end{equation}
for some $C_2\in\mathbb{R}$. Note that the desired conclusion could now be obtained as a direct consequence of Theorem B in \cite{MHSS}. However, for the sake of completeness, we prefer to provide here a direct argument based on the proof of Theorem 10.1 in \cite{CoMi} (see also \cite{Mant}).

Equations \eqref{Rel2} and \eqref{Rel3} are the key geometric identities to prove our assertion. Indeed, let $p\in\Sigma$, $\{e_i\}_{i=1}^{m}$ an orthonormal frame in $T_p\Sigma$ that diagonalizes $\mathbf{A}_{p}$, then
\[
\ \mathbf{A}(e_{i}, e_{j})=\lambda_{i}\delta_{ij}.
\]
Reasoning as in \cite{CoMi}, we hence have by \eqref{Rel2} that
\begin{itemize}
	\item[(i)] For every $k$, there exists $\alpha_{k}$ such that $\nabla_{e_{k}}\left(\mathbf{A}(e_{i}, e_{i})\right)=\alpha_{k}\lambda_{i}$, $i=1\ldots,m$.
	\item[(ii)] If $i\neq j$ then $\nabla_{e_k}\left(\mathbf{A}(e_{i}, e_{j})\right)=0$. 
\end{itemize}
Since $\nabla\mathbf{A}$ is fully symmetric, by Codazzi's equations (ii) implies
\begin{itemize}
	\item[($\tilde{\mathrm{ii}}$)]$\nabla_{e_{k}}\mathbf{A}(e_{i}, e_{j})=0$ unless $i=j=k$. 
\end{itemize}
If $\lambda_{i}\neq0$ and $j\neq i$, then $0=\nabla_{e_{j}}\mathbf{A}(e_{i}, e_{i})=\alpha_{j}\lambda_{i}$, so that $\alpha_{j}=0$.

In particular, if $\mathrm{rk}(A_{p})\geq 2$, then $\alpha_{j}=0$ for every $j$, and thus, by (i), $\left(\nabla\mathbf{A}\right)_{p}=0$.

We now consider two cases depending on the nature of $Ker(\mathbf{A})$.

\textbf{Case 1.} \textit{$Ker\,\mathbf{A}$ is empty everywhere}. $\mathbf{A}$ must have rank at least two everywhere as $m\geq 2$ and, by the reasoning above, $\nabla\mathbf{A}\equiv 0$ on $\Sigma$. According to a theorem by Lawson, \cite{La}, we get that $\Sigma$ should be isometric to $\mathbb{S}^{k}(r)\times\mathbb{R}^{m-k}$ for some $r\geq0$ and $k=2,\ldots,m$. However, none of these examples is a translator.
	\medskip
	
\textbf{Case 2.} \textit{$Ker \mathbf{A}$ is non-empty at some $p\in\Sigma$}. Let $v_{1}(p)\ldots, v_{m-k}(p)\in T_{p}\Sigma\subset\mathbb{R}^{m+1}$ be a family of unit orthonormal tangent vectors spanning such $(m-k)$--dimensional kernel. Then the geodesic $\gamma$ from $p$ in $\Sigma$ with initial velocity $v_{l}(p)$ satisfies
\[
\ \nabla_{\dot{\gamma}}\left(\mathbf{A}(\dot{\gamma}, \cdot)\right)=\frac{1}{H}\left\langle\nabla H, \dot{\gamma}\right\rangle\mathbf{A}(\dot{\gamma}, \cdot),
\]
and by Gronwall's lemma $\mathbf{A}(\dot{\gamma}(s),\cdot)(\gamma(s))=0$ for all $s\in\mathbb{R}$. 

Since $\gamma$ is a geodesic in $\Sigma$ the normal to the curve in $\mathbb{R}^{m+1}$ is also normal to $\Sigma$. Hence, letting $\kappa$ be the curvature of $\gamma$ on $\mathbb{R}^{m+1}$, we have that $\kappa=\left\langle\nu, \frac{d}{ds}\dot\gamma\right\rangle=\mathbf{A}(\dot{\gamma}, \dot{\gamma})=0$, and thus $\gamma$ is a straight line in $\mathbb{R}^{m+1}$. In particular, the whole $(m-k)$--dimensional affine subspace $p+S(p)\subset\mathbb{R}^{m+1}$ is contained in $\Sigma$, where we set $S(p)=\left\langle v_{1}(p), \ldots, v_{m-k}(p)\right\rangle\subset\mathbb{R}^{m+1}$. 

Let now $\sigma$ be a geodesic from $p$ to another point $q\in\Sigma$, parametrized by arc-length and extend by parallel transport the vectors $v_{l}$ along $\sigma$. Then
\[
\ \nabla_{\dot{\sigma}}\left(\mathbf{A}(v_{l}, \cdot)\right)=\frac{1}{H}\left\langle\nabla H, \dot{\sigma}\right\rangle\mathbf{A}(v_{l}, \cdot).
\]
Again by Gronwall's lemma we get that $\mathbf{A}(v_{l}(s), \cdot)=0$ for all $s\in\mathbb{R}$, and in particular $v_{l}(q)$ is contained in the kernel of $\mathbf{A}$ at $q\in \Sigma$. Hence the kernel of $\mathbf{A}$ has constant dimension $m-k$ with $0<k<m$ (as $|H|>0$) at every $q\in \Sigma$ and all the affine $(m-k)$--dimensional subspaces $q+S(q)\subset\mathbb{R}^{m+1}$ are contained in $\Sigma$. Since also $rk(\mathbf{A})$ is constant on $\Sigma$ we conclude as follows.

If $\mathrm{rk}(\mathbf{A})\equiv 1$, then $\Sigma$ is invariant under the isometric translations in the $(m-1)$--dimensional subspace spanned by a global orthonormal frame for $\mathrm{Ker}(\mathbf{A})$. Therefore $\Sigma$ is a product of a curve $\Gamma\subset\mathbb{R}^2$ and this $(m-1)$--dimensional subspace. The curve $\Gamma$ has to be a smooth translator curve in $\mathbb{R}^2$ with $|\mathbf{H}|>0$, hence the grim reaper. Thus $\Sigma$ should be isometric to the grim reaper cylinder. On the other hand our integrability assumption $|\mathbf{A}|\in L^2(\Sigma_f)$ is not met by this hypersurface, leading to a contradiction. 

If $rk(\mathbf{A})\equiv c\geq 2$ we get $\nabla\mathbf{A}\equiv 0$ on $\Sigma$ and hence a contradiction reasoning exactly as in Case 1.
\end{proof}

\begin{proof}(of Theorem \ref{MainA1}) A simple application of the maximum principle to \eqref{EqMeanCurv} gives that either $|H|>0$ or $H\equiv 0$. In this latter case $|\mathbf{A}|=|\mathbf{\Phi}|\in L^{2}(\Sigma_f)$ and, reasoning as in Section \ref{intro}, we get that $\Sigma$ is a translator hyperplane. Therefore, assume by now that $H>0$. The case $H<0$ can be treated exactly in the same way substituting $-H$ to $H$ in the reasoning below. By \eqref{EqSimonTraceless} we have that
 \begin{equation}\label{EqVanPhi}
  |\mathbf{\Phi}|\left[\Delta_f|\mathbf{\Phi}|+|\mathbf{A}|^2||\mathbf{\Phi}|\right]\geq 0.
 \end{equation}
Hence we can apply Lemma \ref{Liouville} as in the proof of the previous theorem, with the choices $u=|\mathbf{\Phi}|$, $v=H$, $a(x)=|\mathbf{A}|^2$ and $\delta=\beta=1$. We thus deduce that
\[
 |\mathbf{\Phi}|=DH,
\]
for some constant $D\geq 0$, and hence, by \eqref{NormPhi2},
\begin{equation}\label{Rel4}
 |\mathbf{A}|=\sqrt{D^2+\frac{1}{m}}H.
\end{equation}
Combining \eqref{Rel4} and \eqref{EqMeanCurv}, we deduce that \eqref{Rel2} holds. The conclusion follows now reasoning as in the proof of Theorem \ref{MainA} and noting that also the assumption $|\mathbf{\Phi}|\in L^2(\Sigma_f)$ is not met by a grim reaper cylinder.
\end{proof}

\section{Weighted Sobolev inequalities for translators}\label{SectSobIneq}

The aim of this section is to obtain the validity of a weighted $L^1$ Sobolev inequality on translators of the mean curvature flow. This will be essential in the proof of the topological results in Section \ref{SectTopCons}.

\medskip
The first result establishes the validity on translators (without any further assumption) of the following $(m+1)$--dimensional weighted $L^1$ Sobolev inequality. An essential tool in the proof, as we shall see, is the bijective correspondence found out by K. Smoczyk, \cite{Smoc}, between translators and minimal hypersurfaces in a suitable warped product.

\begin{lemma}\label{Thm_WeightSobL1}
Let $x:\Sigma^m\to \mathbb{R}^{m+1}$ be a translator.
Let $h$ be a non--negative compactly supported $C^1$ function on $\Sigma$. Then
\begin{equation}\label{SobL1}
\left[\int_{\Sigma}h^{\frac{m+1}{m}}e^{-f}d\mathrm{vol}_{\Sigma}\right]^{\frac{m}{m+1}}
\leq C\int_{\Sigma}|\nabla h|e^{-f}d\mathrm{vol}_{\Sigma}.
\end{equation}
\end{lemma}

\begin{proof}
Let $x:\Sigma^m \to \mathbb{R}^{m+1}$ be a translator and let $f(p)=-\left\langle p,v\right\rangle$. We consider the warped product $\hat{N}=\mathbb{R}^{m+1}\times_{e^{-f}}\mathbb{T}$, where  $\mathbb{T}=\mathbb{T}^{1}=\mathbb{R}/\mathbb{Z}$, so that $\mathrm{vol}(\mathbb{T})=1$. Recall that the warped product metric is given by $g_{\hat{N}}=\left\langle\,,\,\right\rangle_{\mathbb{R}^{m+1}}+e^{-2f}dt^2$. In the rest of the proof $\{E_{i}\}_{i=1}^{m+2}$ will be an orthonormal frame at $(p, t)\in \hat{M}$ such that $\{E_{j}\}_{j=1}^{m+1}$ is a local orthonormal frame at $p\in \mathbb{R}^{m+1}$ and $E_{m+2}=e^{f}\frac{\partial}{\partial t}\in T_{t}\mathbb{T}$. Furthermore, we will denote by $\overline{(\cdot)}$ and by $\hat{(\cdot)}$ respectively, the geometric quantities of $\mathbb{R}^{m+1}$ and of $\hat{N}$.
It was proved in \cite{Smoc} (see also \cite{AS}) that, letting
\[
 \hat{x}:\hat{\Sigma}:=\Sigma\times\mathbb{T}\to\hat{N},
\]
it holds that
\[
 \mathbf{H}_{\hat{\Sigma}}(p, t)=(\mathbf{H}_{\Sigma}(p)+\overline{\nabla}f(p), 0).
\]
In particular $\Sigma$ is a translator if and only if $\tilde{\Sigma}$ is a complete minimal hypersurface in $\hat{N}$. Moreover, a direct computation (see e.g. Proposition 2.2. in \cite{Smoc}) shows that the Riemann curvature tensor of $\hat{N}$ satisfies
\begin{align*}
 \,\hat{R}_{ijkl}=&\overline{R}_{ijkl}\equiv 0,\\
 \,\hat{R}_{m+2\,j\,m+2\,j}=&\,\overline{\mathrm{Hess}}f(E_{i}, E_{i})-\left\langle\,\overline{\nabla} f, E_{i}\right\rangle^2,\\
 =&-\left\langle\,\overline{\nabla} f, E_{i}\right\rangle^2\leq 0,\\
 \,\hat{R}_{m+2\,m+2\,m+2\,m+2}=&0,
\end{align*}
with $i,j,k,l=1,\ldots,m+1$.
In particular, $\hat{N}$ is a Cartan--Hadamard manifold and, using a result of D. Hoffmann and J. Spruck, \cite{HoSp}, we 
deduce that on the minimal hypersurface $\hat{\Sigma}$ the $L^1$--Sobolev inequality
\[
 \left[\int_{\hat{\Sigma}}\hat{h}^{\frac{m+1}{m}}d\mathrm{vol}_{\hat{\Sigma}}\right]^{\frac{m}{m+1}}\leq C\int_{\hat{\Sigma}}|\nabla \hat{h}|d\mathrm{vol}_{\hat{\Sigma}}\quad\quad\forall\,0\leq\hat{h}\in C_{c}^{1}(\hat{\Sigma})
\]
holds.
Now let $0\leq h\in C_{c}^{1}(\Sigma)$ and set $\hat{h}(x,t)=h(x)$. Then
\begin{align*}\label{WeightSobL1Proof}
\left[\int_{\Sigma}h^{\frac{m+1}{m}}e^{-f}d\mathrm{vol}_{\Sigma}\right]^{\frac{m}{m+1}} =&\left[\int_{\hat{\Sigma}}\hat{h}^{\frac{m+1}{m}}d\mathrm{vol}_{\hat{\Sigma}}\right]^{\frac{m}{m+1}}\\
\leq& C\int_{\hat{\Sigma}}|\nabla \hat{h}|d\mathrm{vol}_{\hat{\Sigma}}= C\int_{\Sigma}|\nabla h|e^{-f}d\mathrm{vol}_{\Sigma},
\end{align*}
proving \eqref{SobL1}.
\end{proof}

As a consequence of Lemma \ref{Thm_WeightSobL1}, applying \eqref{SobL1} to the function $h=u^{\frac{2m}{m-1}}$ and using H\"older's inequality, we deduce the validity of the following weighted $L^2$ Sobolev inequality
\begin{equation}\label{SobL2}
 \left[\int_{\Sigma}u^{\frac{2(m+1)}{m-1}}e^{-f}d\mathrm{vol}_{\Sigma}\right]^{\frac{m-1}{m+1}}\leq \left(\frac{2C m}{m-1}\right)^2\int_{\Sigma}|\nabla u|^2e^{-f}d\mathrm{vol}_{\Sigma}.
\end{equation}

Moreover, adapting to the weighted setting a well--known result due to H. D. Cao, Y. Shen, and S. Zhu, \cite{CSZ} (see also Lemma 7.13 in \cite{PRS-Book}, and \cite{BK} for a version of this result in the more general setting of metric measure spaces), we have that if a weighted manifold $M_f$ satisfies for some $0\leq\alpha<1$ the inequality
\begin{equation}\label{eq_WSobL2_alpha}
\left[\int_{M}h^{\frac{2}{1-\alpha}}e^{-f}d\mathrm{vol}_{M}\right]^{1-\alpha}
\leq S(\alpha)^2\int_{M}|\nabla h|^2e^{-f}d\mathrm{vol}_{M},
\end{equation}
for some positive constant $S(\alpha)$ and for every $h \in C_{c}^{\infty}\left(M\right)$,
then every end of $M$ either is non--$f$--parabolic or it has finite $f$--volume.
Note that, by \eqref{SobL2}, this result in particular applies to translators.
On the other hand, it is not difficult to prove that 
if on a weighted manifold $M_f$ \eqref{SobL1} holds, then every end of $M$ has infinite $f$--volume. We hence deduce the following

\begin{proposition}\label{propWSobL2}
Let $\Sigma^m\to \mathbb{R}^{m+1}$ be a translator. Then every end of $\Sigma$ is non--$f$--parabolic 
\end{proposition}

However, note that the $(m+1)$--dimensional weighted $L^1$ Sobolev inequality obtained in Lemma \ref{Thm_WeightSobL1} does not allow to prove Theorem \ref{MainC} stated in the introduction. Indeed, in order to obtain an Anderson--type decay estimate we need the usual ($m$--dimensional) inequality. On the other hand, this latter can be obtained with the additional 
assumption that the translator is contained in a halfspace determined by an hyperplane orthogonal to the translating direction $v$. Namely, we have the validity of the following

\begin{lemma}\label{Thm_mWeightSobL1}
Let $x:\Sigma^m\to \mathbb{R}^{m+1}$ be a translator contained in the halfspace 
$\Pi_{v,a}=\{p\in\mathbb{R}^{m+1}: \langle p,v\rangle\geq a\}$, for some $a\in\mathbb{R}$.
Let $h$ be a non--negative compactly supported $C^1$ function on $\Sigma$. Then
\begin{equation}\label{mSobL1}
\left[\int_{\Sigma}h^{\frac{m}{m-1}}e^{-f}d\mathrm{vol}_{\Sigma}\right]^{\frac{m-1}{m}}
\leq D\int_{\Sigma}|\nabla h|e^{-f}d\mathrm{vol}_{\Sigma}.
\end{equation}
\end{lemma}

\begin{proof}
Let $(M^{m+1}, g)$ be an $(m+1)$--dimensional
Riemannian manifold and $x:\Sigma^m\rightarrow M^{m+1}$ be an immersion. Given a smooth function $f\in C^{\infty}(M)$ we let
$\widetilde{g}=e^{−\frac{2f}{m}}g$ denote a conformal change of metric on $M$. Then $x$ may induce
two isometric immersions of $\Sigma$, that is  $(\Sigma,x^*g)\rightarrow (M, g)$ and 
$(\Sigma, x^*\widetilde{g})\rightarrow (M, \widetilde{g})$ respectively. Moreover
$(\Sigma, x^*\widetilde{g})$ is minimal in $(M, \widetilde{g})$ if and only if $(\Sigma,x^*g)$ is $f$--minimal in
$(M, g)$. This simply follows by the variational formula for the area functional, keeping in mind that 
$d\mathrm{vol}_{\widetilde{\Sigma}}=e^{-f}d\mathrm{vol}_{\Sigma}$.

Assume now that $x:\Sigma^m \to \mathbb{R}^{m+1}$ is a translator for the mean curvature flow and let 
$f(p)=-\left\langle p,v\right\rangle$. 
Denote by $\langle\ ,\ \rangle_{\mathbb{R}^{m+1}}$ the standard metric on $\mathbb{R}^{m+1}$ and consider on $\mathbb{R}^{m+1}$
the conformal metric $\widetilde{\langle\ ,\ \rangle}
:=e^{-\frac{2f}{m}}\langle\ ,\ \rangle_{\mathbb{R}^{m+1}}$. Then, the previous
reasoning implies that $(\Sigma, x^*(\langle\ ,\ \rangle_{\mathbb{R}^{m+1}}))$ is a translator in 
$(\mathbb{R}^{m+1}, \langle\ ,\ \rangle_{\mathbb{R}^{m+1}})$ if and only if 
$\widetilde{\Sigma}:=(\Sigma, x^*(\widetilde{\langle\ ,\ \rangle}))$
is a minimal hypersurface in $(\mathbb{R}^{m+1}, \widetilde{\langle\ ,\ \rangle})$. 

Using the expression for the curvature tensor of a Riemannian manifold under a conformal change 
(see e.g. Section 1.J in \cite{B}) it is not 
difficult to prove that the curvature tensor $\widetilde{R}$ of $(\mathbb{R}^{m+1}, \widetilde{\langle\ ,\ \rangle})$
satisfies
\[
\widetilde{R}_{ijij}=\frac{e^{-\frac{2f}{m}}}{m^2}(\langle\overline{\nabla}f, e_i\rangle^2+\langle\overline{\nabla}f, e_j\rangle^2-|\overline{\nabla} f|^2)\leq 0,
\]
where $\{e_j\}_{j=1}^{m+1}$ denotes the standard orthonormal basis in $\mathbb{R}^{m+1}$.

In particular $(\mathbb{R}^{m+1}, \widetilde{\langle\ ,\ \rangle})$ is a Cartan--Hadamard manifold and hence, by \cite{HoSp}, on the minimal hypersurface $\widetilde{\Sigma}$ the $L^1$--Sobolev inequality
\[
 \left(\int_{\Sigma}h^{\frac{m}{m-1}}d\mathrm{vol}_{\tilde{\Sigma}}\right)^{\frac{m-1}{m}}
 \leq C\int_{\Sigma}\widetilde{|\widetilde{\nabla} h|}d\mathrm{vol}_{\widetilde{\Sigma}}\quad\quad\forall\,0\leq h\in C_{c}^{1}(\Sigma)
\]
holds, where we denoted by $\widetilde{|\cdot|}$ and $\widetilde{\nabla}$ the norm and the Levi--Civita connection on $\widetilde{\Sigma}$ 
respectively. The conclusion is now straightforward keeping in mind that under a conformal change of the metric, the volume form and the norm of the gradient of a given function satisfy
\begin{align*}
d\mathrm{vol}_{\widetilde{\Sigma}}&=e^{-f}d\mathrm{vol}_{\Sigma}\\ 
\widetilde{|\widetilde{\nabla} h|}&=e^{\frac{f}{m}}|\nabla h|\leq e^{\frac{a}{m}}|\nabla h|.
\end{align*}
\end{proof}
 
\section{Topological consequences}\label{SectTopCons}

Recall that, in the non--weighted setting, there is a well known connection, developed mainly by Li and Tam, \cite{LT}, between the dimension of the space of $L^2$--harmonic forms, 
the number of non--parabolic ends and the Morse index of the Schr\"odinger operator
$L=-\Delta-a$, where the function $a$ is the smallest eigenvalue of the Ricci tensor.

In \cite{IR} it is shown that a similar connection remains valid also in the weighted setting. Indeed, adapting Li--Tam theory, it is not difficult to prove that given any relatively compact domain $D$ on a weighted manifold $M_f$, the number of non--$f$--parabolic ends of $M_f$ with respect to $D$ is bounded from above by the dimension of the space of bounded $f$--harmonic functions with finite Dirichlet weighted integral on $M_f$ (see Theorem 4 in \cite{IR}). In particular, suppose that $M_f$ has $N(D)>1$ non--$f$--parabolic ends $E_{A}$, $A=1,\ldots,N(D)$ (with respect to the relatively compact domain $D$). Then one proves that for every $A$ there exists a non--constant bounded $f$--harmonic function  $g_{A}:M_{f}\to\mathbb{R}$, with finite Dirichlet weighted integral, such that $\sup_{E_{A}}g_{A}=1$ and $\inf_{E_{B}}g_{A}=0$, for $B\neq A$. By construction, these functions turn out to be linearly independent. Here we recall, for the sake of completeness, that an end $E$ of $M_{f}$ with respect to a relatively compact domain $D\subset\subset M_f$ with smooth boundary is said to be $f$--parabolic if every positive $f$--superharmonic function satisfying $\frac{\partial u}{\partial n}\geq 0$ on $\partial E$, $n$ being the unit outward normal to $\partial E$, is constant. Otherwise the end will be called non--$f$--parabolic. 

Moreover, combining Theorem 4 in \cite{IR} with a weighted version of a result of D. Fischer--Colbrie, \cite{F}, and of an abstract finiteness result (see Theorem 3 in \cite{IR}) one can also prove the following
\begin{proposition}[Corollary 1 in \cite{IR}]\label{prop_coro1}
Let $M_f$ be a non--compact weighted manifold satisfying 
\[
\mathrm{Ric}_f\geq -a(x)
\]
for some $0\leq a\in C^0(M)$, and let $L_f=-\Delta_f-a(x)$. If $L_f$ has finite
Morse index, then $M_f$ has at most finitely many non--$f$--parabolic ends.
\end{proposition}

We are now in position to prove Theorem \ref{MainD} and Theorem \ref{MainB}, stated in Section \ref{intro}.
\begin{proof}(of Theorem \ref{MainD}). Let $x:\Sigma^m\to\mathbb{R}^{m+1}$, $m\geq 2$, be a $f$--stable translator. Note first that, as a consequence of Proposition \ref{propWSobL2}, every end of $\Sigma$ is non--$f$--parabolic. We reason by contradiction and assume that $\Sigma$ has at least two non--$f$--parabolic ends.
Then, as observed above, there exists a non--constant bounded $f$--harmonic function $u$ such that $|\nabla u|\in L^2(\Sigma_f)$.
Recall the well-known Bochner formula for $f$--harmonic functions
\begin{equation}\label{fBochner}
\frac{1}{2}\Delta_{f}|\nabla u|^2=|\mathrm{Hess}(u)|^2+\mathrm{Ric}_{f}(\nabla u, \nabla u).
\end{equation}
Using \eqref{fBochner}, \eqref{CurvBnd} and the Kato inequality, it is not difficult to prove that $|\nabla u|$ satisfies \eqref{ineq_u}.

Moreover, since $\Sigma$ is $f$--stable, there exist a positive function $v$ satisfying \eqref{ineq_deltav} with $\delta=1$ and $a(x)=|\mathbf{A}|^2(x)$. 
Applying Lemma \ref{Liouville}, we deduce that either $|\nabla u|\equiv 0$, contradicting the fact that $u$ is not constant, or $|\nabla u|=C v$, for some positive constant $C$. Furthermore, in this latter case, inequalities \eqref{ineq_u} and \eqref{ineq_deltav} hold with the equality sign. In particular this implies that $$\mathrm{Ric}_{f}(\nabla u, \nabla u)=-|\mathbf{A}|^2|\nabla u|^2.$$

Fix an arbitrary point $p\in \Sigma$ and let $\{E_{i}(p)\}_{i=1}^{m}$ be an orthonormal basis for $T_{p}\Sigma$ such that $E_{1}(p)=\frac{\nabla u}{|\nabla u|}(p)$. Hence at $p$, 
\[
\ \mathrm{tr}(\mathrm{Ric}_{f})=-|\mathbf{A}|^2-\sum_{i=2}^{m}|\mathbf{A}E_{i}|^2.
\]
On the  other hand, combining Gauss' equation with \eqref{Eqf}, we also have
\[
\ \mathrm{tr}(\mathrm{Ric}_{f})=-|\mathbf{A}|^2.
\]
Hence we obtain
\[
\ \sum_{i=2}^{m}|\mathbf{A}E_{i}|^2=0,
\]
from which we deduce that $|\mathbf{A}|^2(p)=H^2(p)$. Since $p$ is arbitrary this implies that $|\mathbf{A}|^2=H^2$ on $\Sigma$. In particular, as a consequence of Gauss' equation, the scalar curvature of $\Sigma$ vanishes identically. Using Corollary 2.4 in \cite{MHSS} we conclude that $\Sigma$ has to be either a grim reaper cylinder or a translator hyperplane. However, this contradicts the assumption that $\Sigma$ has at least two (non--$f$--parabolic) ends.
\end{proof}

\begin{proof}(of Theorem \ref{MainB}).
Recall that, given a translator
$x:\Sigma^m\rightarrow \mathbb{R}^{m+1}$, the Bakry--\'Emery Ricci tensor of $\Sigma$ satisfies the curvature condition \eqref{CurvBnd}. 
Moreover, if $Ind_{f}(\Sigma)<+\infty$, the assumptions of 
Proposition \ref{prop_coro1} are met with the choice $a(x)=|A|^2(x)$
and we can conclude that $\Sigma$ has at most finitely many non--$f$--parabolic ends. 
The conclusion of the theorem is then straightforward since,  
according to Proposition \ref{propWSobL2},
every end of a translator is non--$f$--parabolic. 
\end{proof}

Finally, the third result we aim at proving in this section, namely Theorem \ref{MainC}, 
deals with the connection,
for a translator $\Sigma$ contained in a halfspace $\Pi_{v,\alpha}$,
between the $L^m$--integrability of its second fundamental form and the 
property of having finite topological type, i.e., the existence of a smooth
compact subset $\Omega\subset\subset\Sigma$ such that $\Sigma\backslash\Omega$
is diffeomorphic to the half--cylinder $\partial\Omega\times\lbrack0,+\infty)$.
The proof of this theorem relies on the following general
result that, in the setting of minimal submanifolds of the Euclidean space,
is due to M. Anderson, \cite{A}.
\begin{lemma}\label{lemma_finitetoptype}
Let $x:\Sigma^{m}\rightarrow\mathbb{R}^{m+1}$ be a
complete, non-compact hypersurface satisfying
\begin{equation} \label{unifest}
\sup_{\Sigma\backslash B_{R}^{\Sigma}\left(  o\right)  }\left\vert
\mathbf{A}\right\vert =o\left( R^{-1}\right)  \text{, as }%
R\rightarrow+\infty.
\end{equation}
Then $x$ is proper and $\Sigma$ has finite topological type.
\end{lemma}

\begin{proof}(of Theorem \ref{MainC}). 
Assume that $x:\Sigma^m\rightarrow\mathbb{R}^{m+1}$, $m\geq 3$, is a complete translator contained in a halfspace
$\Pi_{v, a}$, for some $a\in\mathbb{R}$. Then, applying \eqref{mSobL1} to $h=u^{\frac{2(m-1)}{m-2}}$, we get that on $\Sigma$ the weighted 
$L^{2}$-Sobolev inequality%
\[
\left[\int_{\Sigma}u^{\frac{2m}{m-2}}e^{-f}d\mathrm{vol}_{\Sigma}\right]^{\frac{m-2}{m}}
\leq C^2\int_{\Sigma}|\nabla u|^2e^{-f}d\mathrm{vol}_{\Sigma}
\]
holds, for some constant $C>0$ and for every $u \in C_{c}^{\infty}\left(
\Sigma\right)$. 
Now we recall that, using \eqref{EqSimon} and the Kato inequality, the second 
fundamental form of $\Sigma$ satisfies
the Simons--type inequality%
\[
\Delta_{f}\left\vert \mathbf{A}\right\vert +\left\vert \mathbf{A}\right\vert
^{3}\geq0.
\]
Since $|\mathbf{A}| \in L^{m}(\Sigma_f)$, adapting to the weighted setting the arguments in \cite{PV2},  
we are able to obtain the Anderson--type decay estimate \eqref{unifest}.
The desired conclusion is now a direct consequence of Lemma \ref{lemma_finitetoptype}.
\end{proof}

\section{About boundedness properties of translators}\label{SectBndProp}
A simple application of the maximum principle to \eqref{Eqf} gives that $f$ cannot have a local minimum on $\Sigma$. Thus, in particular, there are no compact translators. 

In the non--compact setting, reasoning in a similar way, one can make use of a version of the maximum principle at infinity, known as the Omori--Yau maximum principle for $\Delta_f$, to obtain that a suitable class of translators cannot be bounded in $\mathbb{R}^{m+1}$. The validity of the Omori--Yau maximum principle for $\Delta_f$, is known to be equivalent to the $f$--stochastic completeness of the hypersurface, namely the Markovianity of the natural diffusion process associated with the weighted Laplacian, \cite{PRS-Memoirs}, \cite{PRiS}.

\begin{proposition}\label{prop_nofstoch}
Let $x:\Sigma^m\to\mathbb{R}^{m+1}$ be an $f$--stochastically complete translator. Then $x(\Sigma)$ cannot be contained in a lower halfspace $\Pi^{-}_{v,b}:=\left\{p\in\mathbb{R}^{m+1} \, :\,\left\langle p, v\right\rangle\leq b\right\}$. In particular $x(\Sigma)$ cannot be bounded. 
\end{proposition}
\begin{proof}
Since $x(\Sigma)$ is $f$--stochastically complete, if $u\in C^2(\Sigma)$ is bounded from below we can find a sequence of points $\{p_j\}\subset\Sigma$ such that
\begin{equation*}
\begin{array}{cc}
(i)\,u(p_j)=\inf_{\Sigma}u+\frac{1}{j}&(ii)\,\Delta_f u(p_j)>-\frac{1}{j}.
\end{array}
\end{equation*}
Assume by contradiction that $x(\Sigma)$ is contained in a lower halfspace $\Pi^{-}_{v,b}$. Then
\[
\ \inf_{\Sigma}f>-\infty,
\]
and we can find a sequence $\{p_j\}\subset\Sigma$ such that
\begin{equation*}
\begin{array}{cc}
(i)\,f(p_j)=\inf_{\Sigma}f+\frac{1}{j}&(ii)\,-1=\Delta_f f(p_j)>-\frac{1}{j}.
\end{array}
\end{equation*}
Passing to the limit as $j\to\infty$ in (\textit{ii}) we reach the desired contradiction.
\end{proof}

\begin{remark}
\rm{It is known that $f$--stochastic completeness is guaranteed if the Bakry--\'Emery Ricci curvature is uniformly bounded from below; see \cite{Q1}, \cite{WW}, \cite{PRiS}. In particular, by \eqref{CurvBnd}, a translator $x:\Sigma^m\to\mathbb{R}^{m+1}$ 
is $f$--stochastically complete if $|\mathbf{A}|\in L^{\infty}(\Sigma)$. 

By the way, note that also properly immersed translators $x:\Sigma^m\to\mathbb{R}^{m+1}$ are $f$-stochastically complete. Indeed, by \eqref{Eqx2},  we get that outside a compact set
\[
\ \Delta_f|x|^2\leq2(m+1)|x|^2.
\]
Thus $|x|^2$ is a proper $\lambda$--f--superharmonic function, with $\lambda=2(m+1)>0$. The $f$--stochastic completeness follows then by the standard Khas'minskii test.}
\end{remark}

Assuming, instead of an $L^{\infty}$ condition on $|\mathbf{A}|$, the finiteness of the total curvature, we are still able to deduce a non--existence result for bounded translators.

\begin{proposition}
Let $x:\Sigma^m\to\mathbb{R}^{m+1}$ be a translator satisfying $|\mathbf{A}|\in L^{m}(\Sigma)$. Then $x(\Sigma)$ cannot be contained in a $v$-slab $\left\{p\in\mathbb{R}^{m+1} \, :\,a\leq\left\langle p, v\right\rangle\leq b\right\}$. In particular, $x(\Sigma)$ cannot be bounded. 
\end{proposition}
\begin{proof}
Assume by contradiction that $x(\Sigma)$ is contained in a $v$-slab. Then, $f$ is bounded and there exist
positive constants $C_1,\ C_2$ such that
\begin{equation}\label{EqMeas}
C_1 d\mathrm{vol}_{\Sigma}\leq  e^{-f}d\mathrm{vol}_{\Sigma}\leq C_2 d\mathrm{vol}_{\Sigma}.
\end{equation}
Hence, combining \eqref{EqMeas} with Theorem 2.1 in \cite{MiSi}, we obtain that on $\Sigma$ the weighted
$L^{1}$-Sobolev inequality
\[
\left[\int_{\Sigma}h^{\frac{m}{m-1}}e^{-f}d\mathrm{vol}_{\Sigma}\right]^{\frac{m-1}{m}}
\leq C\left[\int_{\Sigma}\left(|\nabla h|+|H|h\right)e^{-f}d\mathrm{vol}_{\Sigma}\right]
\]
holds, for some constant $C>0$ and for every $0\leq h \in C_{c}^{1}\left(
\Sigma\right)$. Moreover, the integrability assumption on the second 
fundamental form implies that $H\in L^m(\Sigma_f)$ and, using H\"older's inequality, it is not 
difficult to prove that on $\Sigma$ the $L^2$ weighted Sobolev inequality
\[
\left[\int_{\Sigma}h^{\frac{2m}{m-2}}e^{-f}d\mathrm{vol}_{\Sigma}\right]^{\frac{m-2}{m}}
\leq D^2\left[\int_{\Sigma}|\nabla h|^2e^{-f}d\mathrm{vol}_{\Sigma}\right]
\]
holds, for some constant $D>0$ and for every $h \in C_{c}^{1}\left(
\Sigma\right)$.
Reasoning as in the proof of Theorem \ref{MainC} we obtain again 
the Anderson--type decay estimate \eqref{unifest} for $|\mathbf{A}|$ and hence 
the properness of $x$. But then $\Sigma$ has to be $f$--stochastically complete, 
contradicting Proposition \ref{prop_nofstoch}. 
\end{proof}

\begin{acknowledgement*}
The authors have been supported by the ``Gruppo Nazionale per l'Analisi Matematica, la Probabilit\`a e le loro Applicazioni'' (GNAMPA) of the Instituto Nazionale di Alta Matematica (INdAM). We would like to thank the anonymous referee for useful comments.
\end{acknowledgement*}

\bibliographystyle{amsplain}
\bibliography{TranslSol}

\providecommand{\bysame}{\leavevmode\hbox to3em{\hrulefill}\thinspace}
\providecommand{\MR}{\relax\ifhmode\unskip\space\fi MR }
\providecommand{\MRhref}[2]{%
  \href{http://www.ams.org/mathscinet-getitem?mr=#1}{#2}
}
\providecommand{\href}[2]{#2}
\begin{thebibliography}{10}

\bibitem{AW}
S.~J. {Altschuler} and L.~F. {Wu}, \emph{Translating surfaces of the
  non-parametric mean curvature flow with prescribed contact angle}, Calc. Var.
  Partial Differential Equations \textbf{2} (1994), no.~1, 101--111.

\bibitem{A}
M.~T. {Anderson}, \emph{Scalar curvature, metric degenerations and the static
  vacuum {E}instein equations on 3--manifolds}, Geom. Funct. Anal. \textbf{9}
  (1999), no.~2, 855--967.

\bibitem{AS}
C.~{Arezzo} and J.~{Sun}, \emph{Conformal solitons to the mean curvature flow
  and minimal submanifolds}, Math. Nachr. \textbf{286} (2013), no.~8-9,
  772--790.

\bibitem{Bay}
V.~{Bayle}, \emph{Propri\'et\'es de concavit\'e du profil isop\'erim\'etrique
  et applications}, Th\`ese de Doctorat, 2003.

\bibitem{B}
A.~L. {Besse}, \emph{Einstein manifolds}, Ergebnisse der Mathematik und ihrer
  Grenzgebiete (3) [Results in Mathematics and Related Areas (3)], vol.~10,
  Springer-Verlag, Berlin, 1987.

\bibitem{BK}
S.~M. {Buckley} and P.~{Koskela}, \emph{Ends of metric measure spaces and
  {S}obolev inequalities}, Math. Z. \textbf{252} (2006), no.~2, 275--285.

\bibitem{CSZ}
H.-D. {Cao}, Y.~{Shen}, and S.~{Zhu}, \emph{The structure of stable minimal
  hypersurfaces in {${\bf R}^{n+1}$}}, Math. Res. Lett. \textbf{4} (1997),
  no.~5, 637--644.

\bibitem{CMZ-Simon}
X.~{Cheng}, T.~{Mejia}, and D.~{Zhou}, \emph{Simons-type equation for
  {$f$}-minimal hypersurfaces and applications}, J. Geom. Anal. \textbf{25}
  (2015), no.~4, 2667--2686.

\bibitem{ChMeZh1}
\bysame, \emph{Stability and compactness for complete {$f$}-minimal surfaces},
  Trans. Amer. Math. Soc. \textbf{367} (2015), no.~6, 4041--4059.

\bibitem{CSS}
J.~{Clutterbuck}, O.~C. {Schn{\"u}rer}, and F.~{Schulze}, \emph{Stability of
  translating solutions to mean curvature flow}, Calc. Var. Partial
  Differential Equations \textbf{29} (2007), no.~3, 281--293.

\bibitem{CoMi}
T.~H. {Colding} and W.~P. {Minicozzi II}, \emph{Generic mean curvature flow
  {I}; generic singularities}, Ann. of Math \textbf{175} (2012), no.~2,
  755--833.

\bibitem{Esp}
J.~M. {Espinar}, \emph{Manifolds with density, applications and gradient
  {S}chr\"odinger operators}, arXiv:1209.6162v6.

\bibitem{Fan}
E.~M. {Fan}, \emph{Topology of three-manifolds with positive {$P$}-scalar
  curvature}, Proc. Amer. Math. Soc. \textbf{136} (2008), no.~9, 3255--3261.

\bibitem{F}
D.~Fischer-Colbrie, \emph{On complete minimal surfaces with finite {M}orse
  index in three-manifolds}, Invent. Math. \textbf{82} (1985), no.~1, 121--132.

\bibitem{FCS}
D.~{Fischer-Colbrie} and R.~{Schoen}, \emph{The structure of complete stable
  minimal surfaces in {$3$}-manifolds of nonnegative scalar curvature}, Comm.
  Pure Appl. Math. \textbf{33} (1980), no.~2, 199--211.

\bibitem{Grayson}
M.~A. {Grayson}, \emph{The heat equation shrinks embedded plane curves to round
  points}, J. Differential Geom. \textbf{26} (1987), no.~2, 285--314.

\bibitem{Gro}
M.~{Gromov}, \emph{Isoperimetry of waists and concentration of maps}, Geom.
  Funct. Anal. \textbf{13} (2003), no.~1, 178--215.

\bibitem{Has}
R.~{Haslhofer}, \emph{Uniqueness of the bowl soliton}, Geom. Topol. \textbf{19}
  (2015), no.~4, 2393--2406.

\bibitem{Ho}
P.~T. {Ho}, \emph{The structure of $\phi$--stable minimal hypersurfaces in
  manifolds of nonnegative $p$--scalar curvature}, Math. Ann. \textbf{348}
  (2010), no.~2, 319--332.

\bibitem{HoSp}
D.~{Hoffman} and J.~{Spruck}, \emph{Sobolev and isoperimetric inequalities for
  {R}iemannian submanifolds}, Comm. Pure Appl. Math. \textbf{27} (1974),
  715--727.

\bibitem{HS2}
G.~{Huisken} and C.~{Sinestrari}, \emph{Convexity estimates for mean curvature
  flow and singularities of mean convex surfaces}, Acta Math. \textbf{183}
  (1999), no.~1, 45--70.

\bibitem{HS1}
\bysame, \emph{Mean curvature flow singularities for mean convex surfaces},
  Calc. Var. Partial Differential Equations \textbf{8} (1999), no.~1, 1--14.

\bibitem{IR}
D.~{Impera} and M.~{Rimoldi}, \emph{Stability properties and topology at
  infinity of $f$--minimal hypersurfaces}, Geom. Dedicata \textbf{178} (2015),
  no.~1, 21--47.

\bibitem{La}
H.~B. {Lawson}, \emph{Local rigidity theorems for minimal hypersurfaces}, Ann.
  of Math. (2) \textbf{89} (1969), 187--197.

\bibitem{LT}
P.~{Li} and L.-F. {Tam}, \emph{Harmonic functions and the structure of complete
  manifolds}, J. Differential Geom. \textbf{35} (1992), no.~2, 359--383.

\bibitem{Liu}
G.~{Liu}, \emph{Stable weighted minimal surfaces in manifolds with non-negative
  {B}akry-{E}mery {R}icci tensor}, Comm. Anal. Geom. \textbf{21} (2013), no.~5,
  1061--1079.

\bibitem{LMa}
L.~{Ma}, \emph{{$B$}-sub-manifolds and their stability}, Math. Nachr.
  \textbf{279} (2006), no.~13-14, 1597--1601.

\bibitem{MaMi}
L.~{Ma} and V.~{Miquel}, \emph{Bernstein theorems for translating solitons of
  hypersurfaces}, arxiv:1405.3042v2.

\bibitem{Mant}
C.~Mantegazza, \emph{Lecture notes on mean curvature flow}, Progress in
  Mathematics, vol. 290, Birkh\"auser/Springer Basel AG, Basel, 2011.

\bibitem{MHSS}
F.~{Mart\'in}, A.~{Savas--Halilaj}, and K.~{Smoczyk}, \emph{On the topology of
  translating solitons of the mean curvature flow}, Calc. Var. Partial
  Differential Equations \textbf{54} (2015), no.~3, 2853--2882.

\bibitem{MiSi}
J.~H. {Michael} and L.~M. {Simon}, \emph{Sobolev and mean-value inequalities on
  generalized submanifolds of {$R^{n}$}}, Comm. Pure Appl. Math. \textbf{26}
  (1973), 361--379.

\bibitem{Ng1}
X.~H. {Nguyen}, \emph{Translating tridents}, Comm. Partial Differential
  Equations \textbf{34} (2009), no.~1-3, 257--280.

\bibitem{Ng2}
\bysame, \emph{Complete embedded self-translating surfaces under mean curvature
  flow}, J. Geom. Anal. \textbf{23} (2013), no.~3, 1379--1426.

\bibitem{PRS-Memoirs}
S.~{Pigola}, M.~{Rigoli}, and A.~G. {Setti}, \emph{Maximum principles on
  {R}iemannian manifolds and applications}, Mem. Amer. Math. Soc. \textbf{174}
  (2005), no.~822, x+99.

\bibitem{PRS-JFA05}
\bysame, \emph{Vanishing theorems on {R}iemannian manifolds, and geometric
  applications}, J. Funct. Anal. \textbf{229} (2005), no.~2, 424--461.

\bibitem{PRS-Book}
\bysame, \emph{Vanishing and finiteness results in geometric analysis},
  Progress in Mathematics, vol. 266, Birkh\"auser Verlag, Basel, 2008, A
  generalization of the {B}ochner technique.

\bibitem{PRiS}
S.~{Pigola}, M.~{Rimoldi}, and A.G. {Setti}, \emph{Remarks on non-compact
  gradient {R}icci solitons}, Math. Z. \textbf{268} (2011), 777--790.

\bibitem{PV2}
S.~{Pigola} and G.~{Veronelli}, \emph{Uniform decay estimates for finite-energy
  solutions of semi-linear elliptic inequalities and geometric applications},
  Differential Geom. Appl. \textbf{29} (2011), no.~1, 35--54.

\bibitem{PV1}
\bysame, \emph{Remarks on ${L}^{p}$-vanishing results in geometric analysis.},
  Int. J. Math. \textbf{23} (2012), 1250008, 18 pp.

\bibitem{Q1}
Z.~{Qian}, \emph{On conservation of probability and the {F}eller property},
  Ann. Probab. \textbf{24} (1996), no.~1, 280--292.

\bibitem{R1}
M.~{Rimoldi}, \emph{On a classification theorem for self--shrinkers}, Proc.
  Amer. Math. Soc. \textbf{142} (2014), no.~10, 3605--3613.

\bibitem{Sh}
L.~{Shahriyari}, \emph{Translating graphs by mean curvature flow}, ProQuest
  LLC, Ann Arbor, MI, 2013, Thesis (Ph.D.)--The Johns Hopkins University.

\bibitem{Smoc}
K.~{Smoczyk}, \emph{A relation between mean curvature flow solitons and minimal
  submanifolds}, Math. Nachr. \textbf{229} (2001), 175--186.

\bibitem{XJWang}
X.-J. {Wang}, \emph{Convex solutions to the mean curvature flow}, Ann. of Math.
  (2) \textbf{173} (2011), no.~3, 1185--1239.

\bibitem{WW}
G.~{Wei} and W.~{Wylie}, \emph{Comparison geometry for the {B}akry-{E}mery
  {R}icci tensor}, J. Differential Geom. \textbf{83} (2009), no.~2, 377--405.

\end{thebibliography}
 
\end{document}